\documentclass[a4paper,10pt]{article}

\usepackage{amsmath}
\usepackage{amssymb}
\usepackage{amsthm}
\usepackage[dvips]{graphicx}
\usepackage[all]{xy}

\theoremstyle{plain}
\newtheorem{them}{Theorem}[section]
\newtheorem{lemma}[them]{Lemma}
\newtheorem{prop}[them]{Proposition}
\newtheorem{coro}[them]{Corollary}
\theoremstyle{definition}
\newtheorem{defi}[them]{Definition}
\newtheorem{exam}[them]{Example}

\newtheorem{conv}[them]{Convention}

\newtheorem*{mthem}{Main Theorem}

\newcommand{\lmap}[3]{#1:#2 \longrightarrow #3}
\newcommand{\map}[3]{#1:#2 \rightarrow #3}
\newcommand{\im}[1]{\text{Im}#1}
\newcommand{\sd}[1]{\mathrm{Sd}(\mathcal{#1})}
\newcommand{\nn}[2]{\overline{N_{#1}}(\mathcal{#2})}
\newcommand{\Hom}{\mathrm{Hom}}

\begin{document}

\title{The Euler characteristic of infinite acyclic categories with filtrations}
\author{Kazunori Noguchi \thanks{noguchi@math.shinshu-u.ac.jp}}
\date{}
\maketitle

\begin{abstract}
The aim of this paper is twofold. One is to give a definition of the Euler characteristic of infinite acyclic categories with filtrations and the other is to prove the invariance of the Euler characteristic under the subdivision of finite categories. 
\end{abstract}

\thispagestyle{empty}
\section{Introduction}

The Euler characteristic of a finite simplicial complex is given by the alternating sum of the numbers of simplices in each dimension. On the other hand, Rota defined the Euler characteristic of finite posets \cite{Rot64}. The relation of these two Euler characteristics is explained by the following commutative diagram:

$$
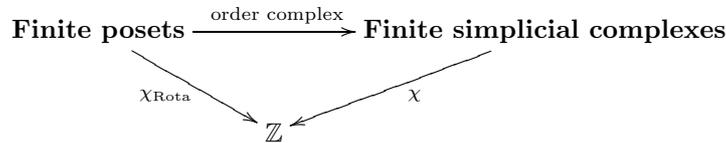
\xymatrix{
{\displaystyle \textbf{Finite posets} } \ar[dr]_{\chi_{\text{Rota}}} \ar[rr]^(0.40){\text{order complex}}&& { \displaystyle \textbf{Finite simplicial complexes} }\ar[dl]^{\chi}\\
&\mathbb{Z}
}$$
Here, the order complex of a finite poset $P$ consists of totally ordered $(n+1)$-subsets of $P$ as its $n$-simplices.

A poset $(P,\le)$ can be regarded as a small category whose set of objects is $P$ and set of the morphisms is the set of pairs $(x,y)$ such that $x\le y$ in $P$. Leinster extended Rota's Euler characteristic \cite{Leia}. He defined the Euler characteristic for finite categories satisfying certain conditions and it agrees with Rota's one when applied to finite posets. Since Leinster's Euler characteristic is also related to another variants such as the cardinality of groupoids \cite{BD} and the rational Euler characteristic of groups \cite{Wal}, it values in the rational numbers. Note that groups are regarded as small categories with one object. The following commutative diagram explains the relation of Leinster's Euler characteristic and Rota's theory
$$
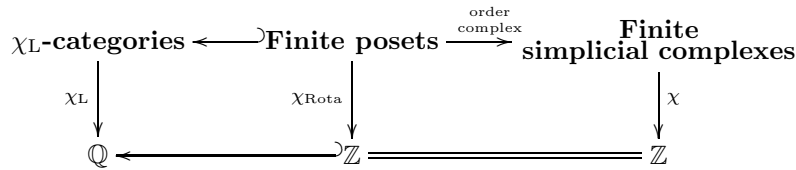
\xymatrix{
\chi_{\text{L}}\textbf{-categories}\ar[d]_{\chi_\text{L}}&{\displaystyle \textbf{Finite posets} } \ar[d]_{\chi_{\text{Rota}}} \ar@{_{(}->}[l] \ar[r]^(0.45){\text{order} \atop \text{complex}}& { \displaystyle \textbf{Finite} \atop \displaystyle \textbf{simplicial complexes} }\ar[d]^{\chi}\\
\mathbb{Q}&\mathbb{Z}\ar@{_{(}->}[l] \ar@{=}[r] &\mathbb{Z}
}$$
where $\chi_{\text{L}}\textbf{-categories}$ denotes the category of finite categories for which Leinster's Euler characteristic can be defined.

On the other hand, the Euler characteristic of simplicial complexes is invariant under the barycentric subdivision. Moreover, small categories also have a notion of the barycentric subdivision. The definition can be found in \cite{DK}, \cite{dH}. It is a functor from the category of small categories to itself 
$$\mathrm{Sd}:\textbf{Small categories}\longrightarrow\textbf{Small categories}$$
and it is homotopy invariant, that is, for any small category $\mathcal{J}$, $\mathcal{J}$ is homotopy equivalent to its barycentiric subdivision Sd$(\mathcal{J})$ after applied the classifying space functor $B$
$$B(\mathrm{Sd}(\mathcal{J})) \simeq B(\mathcal{J}).$$

In this paper, we will investigate the invariance of the Euler characteristic of finite categories under the barycentric subdivision. The difficulty is that the category of finite categories is not closed under the operation of the barycentric subdivision. Any small category which has an endomorphism other than the identity morphisms becomes an infinite category after applied the functor Sd. So we have to extend the Euler characteristic to the class of infinite categories that includes the image of Sd on the category of finite categories.

The barycentric subdivision of small categories is defined by using the notion of \textit{non-degenerate nerve}. Berger and Leinster defined another Euler characteristic $\chi_{\sum}(\mathcal{I})$ of a finite category $\mathcal{I}$, called \textit{series Euler characteristic} \cite{Leib}, in terms of non-degenerate nerves. So it seems to be better to use the series Euler characteristic to investigate the relation between the Euler characteristic of finite categories and the barycentric subdivision of categories. And it coincides Leinster's first Euler characteristic for the important class of finite categories including finite posets and finite groups.

We introduce the Euler characteristic $\chi_{\text{fil}}(\mathcal{A}, \mu)$ of an infinite acyclic category $\mathcal{A}$ with a filtration $\mu$, called $\mathbb{N}$\textit{-filtered acyclic category} (see Definition \ref{n-fil}). For a small category $\mathcal{J}$, its barycentric subdivision Sd$(\mathcal{J})$ is an acyclic category and it naturally has an  $\mathbb{N}$-filtration (see Example \ref{length}).

\begin{mthem}
Let $\mathcal{I}$ be a finite category for which the series Euler characteristic can be defined. Then, $\chi_{\text{fil}}(\mathrm{Sd}(\mathcal{I}), L)$ is also defined and they coincide
$$\chi_{\Sigma}(\mathcal{I})=\chi_{\text{fil}}(\mathrm{Sd}(\mathcal{I}), L),$$
that is, we have the following commutative diagram.
$$\xymatrix{
 \chi_{\sum}\textbf{-categories} \ar[dr]_{\chi_{\sum}} \ar[rr]^(0.45){\mathrm{Sd}}&& \chi_{\text{fil}}\textbf{-categories}   \ar[dl]^{\chi_{\text{fil}}}\\
&\mathbb{Q}
}$$
where $\chi_{\sum} \textbf{-categories}$ denotes the category of finite categories for whicjhthe series Euler  characteristic can be defined and $\chi_{\text{fil}}\textbf{-categories}$ denotes the category of $\mathbb{N}$-filtered acyclic categories for which its Euler characteristic can be defined. 
\end{mthem}

The series Euler characteristic has the invariance under the barycentric subdivision of small categories in the sense of the following commutative diagram.

$$
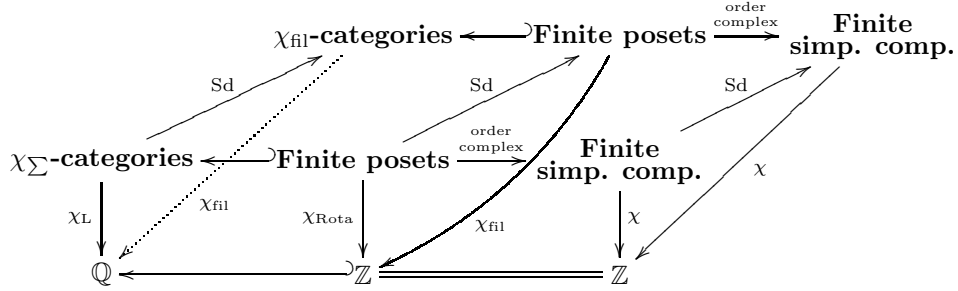
\xymatrix{
&\chi_{\text{fil}}\textbf{-categories} \ar@{.>}[ddl]^(0.65){\chi_\text{fil}}&\textbf{Finite posets}\ar@/^1pc/[ddl]^(0.65){\chi_\text{fil}} \ar@{_{(}->}[l]\ar[r]^(0.5){\text{order} \atop \text{complex}}&{ \displaystyle \textbf{Finite} \atop \displaystyle \textbf{simp. comp.} }\ar[ddl]^{\chi}\\
\chi_{\sum}\textbf{-categories}\ar[ur]^{\mathrm{Sd}}\ar[d]_{\chi_\text{L}}&{\displaystyle \textbf{Finite posets} } \ar[ur]^{\mathrm{Sd}}\ar[d]_{\chi_{\text{Rota}}} \ar@{_{(}->}[l] \ar[r]^(0.5){\text{order} \atop \text{complex}}& { \displaystyle \textbf{Finite} \atop \displaystyle \textbf{simp. comp.} }\ar[d]^{\chi}\ar[ur]^{\mathrm{Sd}}&\\
\mathbb{Q}&\mathbb{Z}\ar@{_{(}->}[l] \ar@{=}[r] &\mathbb{Z}&
}$$

This paper is organized as follows.

In section \ref{Pre}, some conventions and elementary knowledge are recalled. All the keywords in this section can be found in \cite{Koz08}.

In section \ref{sub}, the definition of the barycentric subdivision of small categories is given and some elementary properties are proved. This definition is the one introduced in \cite{DK}.

In section \ref{inv}, we define the Euler characteristic of $\mathbb{N}$-filtered acyclic categories. And we give a proof of the main theorem, theorem \ref{comm}. 

\medskip

\textbf{Acknowledgements.}  I wish to thank Dai Tamaki, Katsuhiko Kuribayashi for very useful discussions and Akihide Hanaki for very helpful suggestions to solve Proposition \ref{Hanaki}. And I also thank Matias Luis del Hoyo who answered my questions about the barycentric subdivision of small categories. 


\section{Preliminaries}\label{Pre}

\begin{conv}
We mean the natural numbers are non-negative integers. So the set of natural numbers $\mathbb{N}$ contains $0$.
$$\mathbb{N}=\{0,1,2,\dots\}$$
$\mathbb{N}$ is regarded as a poset by $0<1<2<\dots$.
\end{conv}

\begin{defi}\label{def}
Define a small category $\mathcal{A}$ to be an \textit{acyclic category} if all the endomorphisms are only identity morphisms and if there exists an arrow $\map{f}{X}{Y}$ such that $X\not = Y$, then there does not exist an arrow $\map{g}{Y}{X}$. 
\end{defi}

\begin{defi}
Let $\mathcal{A}$ be an acyclic category. Define an order on the set of objects of $\mathcal{A}$ by $x \le y$ if there exists a morphism from $x$ to $y$.
\end{defi}

\begin{defi}
Let $\mathcal{J}$ be a small category. The \textit{nerve} $N_*(\mathcal{J})$ of $\mathcal{J}$ is the simplicial set whose set of $n$-simplices $N_n(\mathcal{J})$ is defined as follows \cite{Koz08} \cite{ML98}:
$$N_n( \mathcal{J})=\{ (f_1,f_2,\dots, f_n)\mid\text{each $f_i$ and $f_{i+1}$ are composable} \}$$

 The \textit{non-degenerate nerve} of $\mathcal{J}$ $\nn{*}{J}$ is the $\mathbb{N}$-graded subset of $N_*(\mathcal{J})$ equipped the restrictions of the face operators of $N_*(\mathcal{J})$ 
$$\lmap{d_i}{\nn{n}{J}}{N_{n-1}(\mathcal{J})} $$
and each $\nn{n}{J}$ is defined by the following: 
$$\nn{n}{J}=\{ (f_1,f_2,\dots, f_n) \in N_n( \mathcal{J}) \mid\text{none of $f_i$ is the identity morphism} \}$$
where $\nn{0}{J}$ is defined by $\nn{0}{J}=N_0( \mathcal{J})$. 

\end{defi}

\begin{defi}
Let $\Delta_{\text{inj}}$ be the category whose objects are $$[n]=\{0,1,\dots, n\}$$ for any $0\le n$ and morphisms are order-preserving injections between them. Then, define a \textit{$\Delta$-set} $X$ to be a contravariant functor from $\Delta_{\text{inj}}$ to the category of sets.
$$\lmap{X}{(\Delta_{\text{inj}} )^{\text{op}}}{\textbf{Sets}}$$
\end{defi}

Let $\mathcal{A}$ be an acyclic category. Then, $\nn{*}{A}$ is a $\Delta$-set, since it is closed under the operation of the face operators. Moreover, $\nn{*}{-}$ is a functor from the category of acyclic categories to the category of $\Delta$-sets. The morphisms in the formar are functors $\map{F}{\mathcal{A}}{\mathcal{B}}$ between acyclic categories $\mathcal{A, B}$ satisfying $F(x)<F(y)$ in Ob$(\mathcal{B})$ for $x<y$ in Ob$(\mathcal{A})$. The category of the latter consists of $\Delta$-sets as its objects and natural transformations as its morphisms. 

$$\nn{*}{-}: \textbf{Acyclic categories}\longrightarrow \Delta\textbf{-sets}$$

\begin{exam}
Let $\mathcal{A}$ be an acyclic category and $\map{\mu}{\mathcal{A}}{\mathbb{N}}$ be a functor satisfying $\mu (x)<\mu (y)$ in Ob$(\mathcal{B})$ for $x<y$ in Ob$(\mathcal{A})$. By applying the functor $\nn{*}{-}$ to $\mu$, we have the natural transformation $\map{\nn{*}{\mu}}{\nn{*}{A}}{\nn{*}{\mathbb{N}}}$. Here, each $\nn{n}{\mathbb{N}}$ is the set of proper increasing sequences of $(n+1)$-natural numbers. 
$$\nn{n}{\mathbb{N}}=\{(i_0,i_1,\dots, i_n) \in \mathbb{N}^{n+1}  \mid i_0<i_1<\dots < i_n\}$$
For $$\mathbf{f}=\xymatrix{(x_0\ar[r]^{f_1}&x_1\ar[r]^{f_2}&\dots\ar[r]^{f_n}&x_n)}$$
of $\nn{n}{A}$, $\nn{n}{\mu}(\mathbf{f})=(\mu(x_0),\mu(x_1),\dots, \mu(x_n))$.
\end{exam}

\begin{conv} Let $\mathcal{J}$ be a small category and $X$ be an element of $N_n(\mathcal{J})$. We often denote the length of $X$ by $q_X$, that is, $q_X =n$. Then, $X$ has the following form. 
$$\xymatrix{X=(x_0\ar[r]^(0.6){f_1}&x_1\ar[r]^{f_2}&\dots\ar[r]^{f_{q_X}}&x_{q_X})}$$
\end{conv}


\section{The barycentric subdivision of small categories}\label{sub}

Let us recall the definition of the barycentric subdivision of small categories from \cite{DK}.

\begin{defi}
Let $\mathcal{J}$ be a small category. Then, \textit{the barycentric subdivision} $\mathrm{Sd}(\mathcal{J})$ of $\mathcal{J}$ is a small category whose objects are elements of the non-degenerate nerve of $\mathcal{J}$ and the set of morphisms between $X$ and $Y$ is the quotient set of order-preserving maps $\map{f}{[q_X]}{[q_Y]}$ satisfying $Y\circ f=X$ under the relation defined below. Here, $X$ and $Y$ are regarded as functors from posets $[q_X]$ and $[q_Y]$ to $\mathcal{J}$ respectively. So the condition $Y\circ f=X$ means the commutative diagram 
$$\xymatrix{
&\mathcal{J}&\\
[q_X]\ar[ru]^X\ar[rr]^f&&[q_Y]\ar[lu]_Y
}$$
in the category of small categories.

The equivalence relation is generated by the following relation: Given order-preserving maps $\map{f,g}{[q_X]}{[q_Y]}$ satisfying $Y\circ f=X, Y\circ g=X$ respectively. Define $f\sim g$ if for any $0\le i \le q_X$, $Y(\min\{f(i), g(i)\}\rightarrow \max\{f(i), g(i)\})$ is an identity morphism. Here, $$\min\{f(i), g(i)\}\rightarrow \max\{f(i), g(i)\}$$ is a morphism in $[q_Y]$. The composition in Sd$(\mathcal{J})$ is defined by the composition of order-preserving maps.
\end{defi}

We would like to use the properties stated in \cite{dH}, but the definition above is different from the one defined in \cite{dH}. So we give proofs of them here.

\begin{lemma}
The relation given above is an equivalence relation and compatible with the composition, that is, if $[f]=[f']$ and $[g]=[g']$ and they are composable, then $[g\circ f]=[g'\circ f']$.
\begin{proof}
It is easy to show that the relation is an equivalence relation.

To prove the second statement it suffices to show that if $f\sim f'$, then $g\circ f \sim g\circ f'$ and if $g\sim g'$, then $g\circ f \sim g'\circ f$, but the latter one is clear. Suppose $f \sim f'$ and $g$ is composable with them and we have the following diagram
$$\xymatrix{
&\mathcal{J}& \\
[q_X]\ar[ru]^X\ar@<1ex>[r]^f \ar@<-1ex>[r]_{f'}&[q_Y]\ar[u]_Y\ar[r]^g&[q_Z]\ar[lu]_Z.
}$$
For any $0\le i \le q_X$, we have 
\begin{eqnarray*}
Z(\min\{g\circ f(i), g\circ f'(i)\}\rightarrow \max\{g\circ f(i), g\circ f'(i)\}) & = & \\
Z\circ g(\min\{f(i), f'(i)\}\rightarrow \max\{f(i),f'(i)\})&=&  \\
Y(\min\{f(i), f'(i)\}\rightarrow \max\{f(i),f'(i)\})&=&1.
\end{eqnarray*}
We conclude the relation is compatible with the composition.
\end{proof}
\end{lemma}

\begin{lemma}\label{inj}
Let $\mathcal{J}$ be a small category. For any morphism $\map{[f]}{X}{Y}$ in Sd$(\mathcal{J})$, $\map{f}{[q_X]}{[q_Y]}$ is an injection.
\begin{proof}
Suppose $f$ is not an injection. Then, there exist $i,j$ such that $f(i)=f(j)$ but $i\not =j$. Suppose $i<j$. Then, we have the inequality $i<i+1\le j$. Since $f$ is an order-preserving map, we have $f(i)\le f(i+1)\le f(j)$. Since $f(i)=f(j)$, so $f(i)=f(i+1)$.

For the morphism $i\rightarrow i+1$ in $[q_X]$, we apply the commutative diagram $X=Y\circ f$. Then, we have
\begin{eqnarray*}
X(i\rightarrow i+1) & = & Y\circ f(i\rightarrow i+1)\\
&=& Y(f(i)\rightarrow f(i)) \\
&=&1_{Y(f(i))}.
\end{eqnarray*}
But $X$ is an element of the non-degenerate nerve. So $f$ must be an injection.
\end{proof}
\end{lemma}

\begin{prop}\label{acyclic}
For any small category $\mathcal{J}$, Sd$(\mathcal{J})$ is an acyclic category.
\begin{proof}
Given an endomorphism $\map{[i]}{X}{X}$, $i$ is an order-preserving map from $[q_X]$ to $[q_X]$. Lemma $\ref{inj}$ implies $i$ is an injection. The only order preserving injection from $[q_X]$ to itself is the  identity map. So all the endomorphisms of Sd$(\mathcal{J})$ are identity morphisms.

Let $\map{[f]}{X}{Y}$ be a morphism such that $X\not = Y$. Since $f$ is an injection, $q_X\le q_Y$. If there exists $\map{[g]}{Y}{X}$, then we have $q_Y\le q_X$. Therefore, we obtain $q_X=q_Y$. The same argument above implies $f$ is the identity map, but $X\not = Y$. 

Thus, we conclude that Sd$(\mathcal{J})$ is an acyclic category.
\end{proof}
\end{prop}

\begin{prop}\label{End}
Let $\mathcal{J}$ be a small category. Then, $\mathrm{Sd}(\mathcal{J})$ is a poset if and only if $\mathrm{End}(x)=1$ for any object $x$ of $\mathcal{J}$.
\begin{proof}
Suppose Sd$(\mathcal{J})$ is a poset. If $\map{i}{x}{x}$ is not an identity map, then we have two maps $\map{i_0, i_1}{[0]}{[1]}$ defined by $i_k(0)=k$ for $k=0,1$ and they are morphisms from $x$ to $i$ in Sd$(\mathcal{J})$.
$$\xymatrix{
&\mathcal{J}&\\
[0]\ar[ru]^x\ar@<1ex>[rr]^{i_0}\ar@<-1ex>[rr]_{i_1}&&[1]\ar[lu]_i
}$$
But they are not equivalent. Indeed, we have
\begin{eqnarray*}
i(\min\{i_0(0), i_1(0)\}\rightarrow \max\{i_0(0), i_1(0)\}) & = & i(0\rightarrow 1)\\
&=& i \\
&\not =&1 
\end{eqnarray*}
So we have $\#\text{Hom}_{\mathrm{Sd}(\mathcal{J})}(x,i)\ge 2$ where $\#$ means the cardinality of sets. This contradicts to the fact that Sd($\mathcal{J}$) is a poset. 

Conversely, suppose End$(x)=1$ for all objects $x$ of $\mathcal{J}$. Then, it suffices to show that $\#\text{Hom}_{\mathrm{Sd}(\mathcal{J})}(X,Y)\le 1$ for any $X,Y$. Take two morphisms $\map{[f],[g]}{X}{Y}$. Then, $f$ is equivalent to $g$. Indeed, for any $0\le i\le q_X$, the commutative diagrams $X=Y\circ f=Y\circ g$ imply $X(i)=Y\circ f(i)=Y\circ g(i)$. So we have 
$$Y(\min\{f(i), g(i)\}\rightarrow \max\{f(i), g(i)\}) \in \text{End}(X(i)).$$
Since End$(X(i))=1$, $f$ is equivalent to $g$.
\end{proof}
\end{prop}

\begin{coro}
Let $\mathcal{J}$ be a small category. Then, $\mathrm{Sd}^2(\mathcal{J})$ is a poset.
\begin{proof}
Lemma $\ref{acyclic}$ implies Sd($\mathcal{J}$) is an acyclic category. And Proposition $\ref{End}$ implies $\mathrm{Sd}^2(\mathcal{J})$ is a poset.
\end{proof}
\end{coro}


\section{The invariance of the Euler characteristic}\label{inv}

\begin{defi}\label{n-fil}
Let $\mathcal{A}$ be an acyclic category. A functor $\map{\mu}{\mathcal{A}}{\mathbb{N}}$ satisfying $\mu (x)<\mu (y)$ in Ob$(\mathcal{B})$ for $x<y$ in Ob$(\mathcal{A})$, $\mu (x) < \mu (y)$ is called \textit{an $\mathbb{N}$-filtration of $\mathcal{A}$}. A pair $(\mathcal{A}, \mu )$ is called \textit{an $\mathbb{N}$-filtered acyclic category}.
\end{defi}

\begin{exam}\label{length}
Let $\mathcal{J}$ be a small category. Proposition \ref{acyclic} implies Sd$(\mathcal{J})$ is an acyclic category. The length functor $L$ gives a natural $\mathbb{N}$-filtration to Sd$(\mathcal{J})$ where the functor $L$ is defined by $L(\mathbf{f})=n$ for $\mathbf{f}$ of $\overline{N_n}( \mathcal{J})$. Thus, we obtain an $\mathbb{N}$-filtered acyclic category $(\mathrm{Sd}(\mathcal{J}),L)$.
\end{exam}

We have the following commutative diagram.

$$\xymatrix{
\mathbb{Z}[t] \ar@{^{(}->}[d] \ar@{^{(}->}[r]&\mathbb{Z}[[t]] \ar@{^{(}->}[d]&\\
\mathbb{Q}(t)\ar[r]\ar@{^{(}->}[r]&\mathbb{Q}((t))
}$$
Here, $\mathbb{Z}[t]$ is the polynomial ring with the coefficients in $\mathbb{Z}$ and $\mathbb{Z}[[t]]$ is the ring of formal power series over $\mathbb{Z}$. $\mathbb{Q}(t)$ and $\mathbb{Q}((t))$ are the quotient fields of them respectively. 

\begin{defi}
Let $f(t)$ be a formal power series over $\mathbb{Z}$. If there exists a rational function $g(t)/h(t)$ in $\mathbb{Q}(t)$ such that $f(t)=g(t)/h(t)$ in $\mathbb{Q}((t))$, then define 
$$ f|_{t=-1}= g(-1)/h(-1) \in \mathbb{Q}$$
if $h(-1) \not = 0.$
\end{defi}

\begin{defi}
Let $(\mathcal{A}, \mu)$ be an $\mathbb{N}$-filtered acyclic category. Then, define $\chi_{\text{fil}}(\mathcal{A}, \mu)$ as follows. 

We have the pair of the $\Delta$-set and the natural transformation $$(\overline{N_*}(\mathcal{A}), \overline{N_*}(\mu)).$$ Let
$$\overline{N_i}(\mathcal{A})_n=\{ \mathbf{f}\in \overline{N_i}(\mathcal{A})\mid \max(\overline{N_i}(\mu)(\mathbf{f}))=n  \}$$
for natural numbers $i,n$. Suppose each $\overline{N_i}(\mathcal{A})_n$ is finite and $\overline{N_i}(\mathcal{A})_n$ is an empty-set if $n<i$. Define the formal power series $f_{\chi}(\mathcal{A}, \mu)(t)$ over $\mathbb{Z}$ by

$$f_{\chi}(\mathcal{A}, \mu)(t)=\sum^\infty_{n=0}(-1)^n\left(\sum^n_{i=0}(-1)^i\# \overline{N_i}(\mathcal{A})_n \right)t^n$$
where the symbol $\#$ means the cardinality of sets. And define
$$\chi_{\text{fil}}(\mathcal{A}, \mu)=f_{\chi}(\mathcal{A}, \mu)(t)|_{t=-1}$$
if it exists. 
\end{defi}

\begin{exam}
Equip the poset $\mathbb{N}$ with the identity functor Id as its $\mathbb{N}$-filtration. Then, we have
\begin{eqnarray*}
\overline{N_i}(\mathbb{N})_n&=&\{ \mathbf{m}=(m_0,m_1,\dots, m_i) \in \overline{N_i}(\mathbb{N}) \mid \max(\overline{N_i}(\text{Id})(\mathbf{m}))=n \}\\
&=&\{(m_0,m_1,\dots, m_{i-1}, n) \mid 0\le m_0< m_1< \dots < m_{i-1}<  n) \}
\end{eqnarray*}
for any $i, n$. Therefore, we obtain $\#\overline{N_i}(\mathbb{N})_n$ is $ n \choose i$. Thus, we have 
\begin{eqnarray*}
f_{\chi}(\mathbb{N}, \text{Id})(t)&=&\sum^\infty_{n=0}(-1)^n\left(\sum^n_{i=0}(-1)^i\# \overline{N_i}(\mathbb{N})_n \right)t^n \\
&=&\sum^\infty_{n=0}(-1)^n\left(\sum^n_{i=0}(-1)^i {n \choose i} \right)t^n \\
&=&1.
\end{eqnarray*}
So we have 
$$\chi_{\text{fil}}(\mathbb{N}, \text{Id})=1.$$
Since the classifying space of $\mathbb{N}$ is contractible, from a topological viewpoint, the Euler characteristic of $\mathbb{N}$ should be $1$, too.
\end{exam}


The rest of this section is devoted to the proof of our main theorem. We reduce it to a combinatorial problem in the form of the following Proposition. 
\begin{prop}\label{Hanaki}
Let $n$ be a natural number. Suppose $\sim$ is an equivelence relation on $[n]$ with the property that if $i\sim j$, then $i+1\not = j$ and $i\not = j+1$. Let 
$$A_k^{(n)}=\{(i_0,i_1, \dots , i_k) \in [n]^{k+1}\mid i_0 < \dots <i_k \}$$
and
$$B_k^{(n)}=\{(i_0,i_1, \dots , i_k) \in A_k^{(n)}\mid \exists m \text{ s.t. } i_m \sim i_{m+1}\}$$
and 
$$C_k^{(n)}=(A_k^{(n)}-B_k^{(n)})/\approx$$ 
where $ (i_0,i_1, \dots , i_k) \approx (j_0,j_1, \dots , j_k)$ is defined by $i_m \sim j_m$ for any $m$. Let $A_{-1}^{(n)}=\{ () \} \cong *$, $B_{-1}^{(n)}= \emptyset$ and $ C_{-1}^{(n)}=*$. Let $\beta_k^{(n)}=\#B_k^{(n)}$ and $\gamma_k^{(n)}=\sum_{[x]\in C_k^{(n)}} (\#[x]-1)$. Then, we have
$$\sum^{n-1}_{k=0}(-1)^k(\beta_k^{(n)}+\gamma_k^{(n)})=0.$$
\end{prop}

We assume this Proposition and we are going to complete the proof of the main theorem. The proof of this Proposition is given later.

Let $\mathcal{I}$ be a finite category. We denote 
$$\nn{i}{\sd{J}}_{\mathbf{f}}=\{ \xymatrix{(\mathbf{f}_0\ar[r]^(0.6){\varphi_1}&\mathbf{f}_1\ar[r]^{\varphi_2}&\dots\ar[r]^{\varphi_i}&\mathbf{f}_i)} \in \nn{i}{\sd{J}}\mid \mathbf{f}_i = \mathbf{f} \}$$
for any element $\mathbf{f}$ of $\nn{n}{I}$ and any $i$. Then, we have the equation 
$$1=(-1)^n\sum^n_{i=0}(-1)^i\#\overline{N_i}(\mathrm{Sd}( \mathcal{I}))_{\mathbf{f}}$$
proved in the next theorem. By summing the equations over all elements of $\mathbf{f}$ of $\nn{n}{I}$, we have $$\#\overline{N_n}( \mathcal{I})=(-1)^n\sum^n_{i=0}(-1)^i\#\overline{N_i}(\mathrm{Sd}( \mathcal{I}))_n.$$

We work on the polynomial ring $\mathbb{Z}[s]$ before substituting $s$ by $-1$ to make calculations easy to see.

\begin{them}\label{main}
Let $\mathcal{J}$ be a small category and $n$ be a natural number. Then, for any $\mathbf{f}$ of $\overline{N_n}( \mathcal{J})$, we have
\begin{eqnarray}
\sum^n_{i=0}(\#\overline{N_i}(\mathrm{Sd}( \mathcal{J}))_{\mathbf{f}})s^i-s^n&=&(1+s)P_{\mathbf{f}}(s)\label{XX}
\end{eqnarray}
for some $P_{\mathbf{f}}(s)$ of $\mathbb{Z}[s]$.
\begin{proof}
We will give an inductive proof. 

At $n=0$, the equation holds as $P_{\mathbf{f}} (s)=0$.

Suppose the equation holds for any $k$ less than $n$. Then, we have 
$$\#\overline{N_i}(\mathrm{Sd}( \mathcal{J}))_{\mathbf{f}}=\sum^{n-1}_{k=i-1} \sum_{\mathbf{g}\in \overline{N_k}( \mathcal{J})}\#\overline{N_{i-1}}(\mathrm{Sd}( \mathcal{J}))_{\mathbf{g}} \times \# \Hom_{\mathrm{Sd}( \mathcal{J})}(\mathbf{g},\mathbf{f})$$ for any $1\le i \le n.$ The left hand side of $(\ref{XX})$ is 
\begin{eqnarray}
&\displaystyle  1-s^n+\sum^n_{i=1}(\#\overline{N_i}(\mathrm{Sd}( \mathcal{J}))_{\mathbf{f}})s^i \\
=&\displaystyle  1-s^n + \sum_{i=1}^{n} \{ \sum^{n-1}_{k=i-1} \sum_{\mathbf{g}\in \overline{N_k}( \mathcal{J})}\# \overline{N_{i-1}}(\mathrm{Sd}( \mathcal{J}))_{\mathbf{g}} \times \# \text{Hom}_{\mathrm{Sd}( \mathcal{J})}(\mathbf{g},\mathbf{f})\}s^i \notag \\
=&\displaystyle  1-s^n + \sum_{i=0}^{n-1} \{ \sum^{n-1}_{k=i} \sum_{\mathbf{g}\in \overline{N_k}( \mathcal{J})}\# \overline{N_{i}}(\mathrm{Sd}( \mathcal{J}))_{\mathbf{g}} \times \# \text{Hom}_{\mathrm{Sd}( \mathcal{J})}(\mathbf{g},\mathbf{f}) \} s^{i+1} \notag \\
=&\displaystyle  1-s^n +\sum^{n-1}_{k=0} \sum_{\mathbf{g}\in \overline{N_k}( \mathcal{J})} \# \text{Hom}_{\mathrm{Sd}( \mathcal{J})}(\mathbf{g},\mathbf{f})s(\sum^k_{i=0} \#\overline{N_i}(\mathrm{Sd}( \mathcal{J}))_{\mathbf{g}}s^i) \label{kogatsura}
\end{eqnarray}
The assumption implies $(\ref{kogatsura})$ is equal to
\begin{eqnarray}
&\displaystyle  1-s^n +\sum^{n-1}_{k=0} \# \text{Hom}_{\mathrm{Sd}( \mathcal{J})}(\mathbf{g},\mathbf{f})s( \sum_{\mathbf{g}\in \overline{N_k}( \mathcal{J})} (1+s)P^k_{\mathbf{g}}(s) + s^k ) \notag\\
=&\displaystyle  1-s^n +\sum^{n-1}_{k=0} \sum_{\mathbf{g}\in \overline{N_k}( \mathcal{J})} \# \text{Hom}_{\mathrm{Sd}( \mathcal{J})}(\mathbf{g},\mathbf{f})s^{k+1} +  (1+s)P_{\mathbf{f}}(s) \label{YY}
\end{eqnarray}
for some $P_{\mathbf{f}}(s)$ in $\mathbb{Z}[s].$
 
Note that $\text{Hom}_{\mathrm{Sd}( \mathcal{J})} (\mathbf{g},\mathbf{f})$ is an empty-set for most $\mathbf{g}$ in $\overline{N_i}( \mathcal{J})$, since they have to satisfy the following commutative diagram.
 $$\xymatrix{
&\mathcal{J}&\\
[i]\ar[ru]^{\mathbf{g}}\ar[rr]^\varphi&&[n]\ar[lu]_{\mathbf{f}}
}$$
Since $\mathbf{g}=\mathbf{f}\circ \varphi$ and $\mathbf{f}$ is fixed, $\mathbf{g}$ is exactly determined by $\varphi$. Lemma \ref{inj} implies $\varphi$ is an injection. There are $n+1 \choose i+1$ order-preserving injections from $[i]$ to $[n]$. We can express the set of such injections by the set of proper increasing sequences of $(i+1)$ elements in $[n]$.
$$\text{Inj}([i],[n])\cong \{ (k_0, \dots, k_i)\in [n]^{i+1}\mid k_0<\dots <k_i\}$$ 
We apply Proposition \ref{Hanaki} to the set in the right hand side which corresponds to $A_i^{(n)}$ in the Proposition. Define an equivalence relation $\sim_{\mathbf{f}}$ on $[n]$ by $j\sim_{\mathbf{f}} j'$ if $\mathbf{f}(\min\{j,j'\}\rightarrow \max\{j,j'\})$ is an identity morphism. Then, we have a one-to-one correspondence between 
$$\{\varphi \in \text{Inj}([i],[n])\mid \exists j \text{ s.t. } \mathbf{f}(\varphi(j)\rightarrow \varphi(j+1))=1\}$$
and 
$$\{(k_0,k_1, \dots , k_i) \in A_k^{(n)}\mid \exists m \text{ s.t. } i_m \sim_{\textbf{f}} i_{m+1}\}.$$The latter one has been denoted by $B_i^{(n)}$ in Proposition \ref{Hanaki} in the case that the equivalence relation is $\sim_{\textbf{f}}$.
We obtain 
\begin{eqnarray*}
\sum_{\mathbf{g} \in \overline{N_i}( \mathcal{J} )}\#\{\varphi: \mathbf{g}\rightarrow \mathbf{f}\}&=&{n+1 \choose i+1}-\# B^{(n)}_i \\
&=& {n+1 \choose i+1}-\beta^{(n)}_i.
\end{eqnarray*}

The remainder of the proof is devoted to counting the number of morphisms eliminated by the equivalence relation in the definition of the barycentric subdivision. The number is expressed by $\gamma^{(n)}_i$ of Proposition \ref{Hanaki}. We obtain 
\begin{eqnarray*}
\sum_{\mathbf{g}\in \overline{N_i}( \mathcal{J} )} \#\text{Hom}_{\mathrm{Sd}( \mathcal{J})} (\mathbf{g},\mathbf{f})&=&{n+1 \choose i+1}-\beta^{(n)}_i-\gamma^{(n)}_i.
\end{eqnarray*}

So the right hand side of (\ref{YY}) is

\begin{eqnarray*}
&\displaystyle 1+\sum^{n-1}_{i=0}\{ {n+1 \choose i+1}-\beta^{(n)}_i-\gamma^{(n)}_i \}s^{i+1}-s^n+(1+s)P_{\mathbf{f}} (s)\\
=&\displaystyle 1+\sum^{n-1}_{i=0}{n+1 \choose i+1}s^{i+1}-s\sum^{n-1}_{i=0}(\beta^{(n)}_i+\gamma^{(n)}_i)s^i-s^n+(1+s)P_{\mathbf{f}} (s)\\
=&\displaystyle \sum^{n}_{i=0}{n+1 \choose i}s^i -s\sum^{n-1}_{i=0}(\beta^{(n)}_i+\gamma^{(n)}_i)s^i-s^n+(1+s)P_{\mathbf{f}} (s)\\
=&\displaystyle (1+s)^{n+1}-s^{n+1}-s^n -s\sum^{n-1}_{i=0}(\beta^{(n)}_i+\gamma^{(n)}_i)s^i+(1+s)P_{\mathbf{f}} (s)\\
=&\displaystyle -s\sum^{n-1}_{i=0}(\beta^{(n)}_i+\gamma^{(n)}_i)s^i+(1+s)Q_{\mathbf{f}} (s)
\end{eqnarray*}
for some $Q_{\mathbf{f}} (s)$ in $\mathbb{Z}[s]$. Here, note that $\sim_{\mathbf{f}}$ satisfies the property that if $k\sim_{\mathbf{f}}k'$, then $k\not = k'+1$ and $k+1\not = k'$ since $\mathbf{f}$ is non-degenerate. So we can apply Proposition \ref{Hanaki}. Therefore, $\sum^{n-1}_{i=0}(\beta^{(n)}_i+\gamma^{(n)}_i)s^i$ can be factored by $(1+s)$. 
\end{proof}
\end{them}

\begin{coro}\label{function}
Let $\mathcal{I}$ be a finite category and $n$ be a natural number. Then,
$$\#\overline{N_n}( \mathcal{I})=(-1)^n\sum^n_{i=0}(-1)^i\#\overline{N_i}(\mathrm{Sd}( \mathcal{I}))_n.$$
\begin{proof}
Theorem \ref{main} implies 
$$\sum^n_{i=0}(\#\overline{N_i}(\mathrm{Sd}( \mathcal{I}))_{\mathbf{f}})s^i-s^n=(1+s)P_{\mathbf{f}}(s)$$
for all $\mathbf{f}$ of $\overline{N_n}( \mathcal{I})$. Sum this equation over all $\mathbf{f}$ of $\overline{N_n}( \mathcal{I})$. Then, we have 

\begin{eqnarray*}
\sum_{\mathbf{f} \in \overline{N_n}( \mathcal{I})} \sum^n_{i=0}(\#\overline{N_i}(\text{Sd}( \mathcal{I}))_{\mathbf{f}})s^i- \sum_{\mathbf{f} \in \overline{N_n}( \mathcal{I})} s^n&=&(1+s) \sum_{\mathbf{f} \in \overline{N_n}( \mathcal{I})} P_{\mathbf{f}}(s) \\
\sum^n_{i=0}(\#\overline{N_i}(\text{Sd}( \mathcal{I}))_n s^i-\#\overline{N_n}(\mathcal{I})s^n &=& (1+s) \sum_{\mathbf{f} \in \overline{N_n}( \mathcal{I})} P_{\mathbf{f}}(s). \\
\end{eqnarray*}
At $s=-1$, we have 
$$\#\overline{N_i}( \mathcal{I})=(-1)^n\sum^n_{i=0}(-1)^i\#\overline{N_n}(\text{Sd}( \mathcal{I}))_n.$$
\end{proof}
\end{coro}

\begin{them}\label{comm}
Let $\mathcal{I}$ be a finite category for which the series Euler characteristic can be defined. Then, $\chi_{\text{fil}}(\mathrm{Sd}(\mathcal{I}), L)$ is also defined and they coincide
$$\chi_{\Sigma}(\mathcal{I})=\chi_{\text{fil}}(\mathrm{Sd}(\mathcal{I}), L)$$
where $L$ is the length functor defined in Example \ref{length}.
\begin{proof}
Recall that the series Euler characteristic of $\mathcal{I}$ is defined by 
$$\chi_{\sum}(\mathcal{I})=(\sum_{n=0}^\infty \#\overline{N_n}( \mathcal{I})t^n )|_{t=-1}.$$
Corollary \ref{function} implies 
\begin{eqnarray*}
\sum_{n=0}^\infty \#\overline{N_n}( \mathcal{I})t^n &=& \sum_{n=0}^\infty \left( (-1)^n\sum^n_{i=0}(-1)^i\#\overline{N_i}(\text{Sd}( \mathcal{I}))_n\right) t^n .
\end{eqnarray*}
Since they coincide as formal power series, if $\chi_{\sum}(\mathcal{I})$ exists, the other also does. We obtain 
$$\chi_{\text{fil}}(\mathrm{Sd}(\mathcal{I}), L).$$
\end{proof}
\end{them}

We have completed the proof of our main theorem under the assumption of Proposition \ref{Hanaki}, whose proof is given below.

\begin{proof}[Proof of Proposition \ref{Hanaki}]
We have 
\begin{eqnarray*}
\gamma_k^{(n)} & = & \sum_{[x]\in C_k^{(n)}} (\#[x]-1)\\
&=& \sum_{[x]\in C_k^{(n)}} \#[x]- \sum_{[x]\in C_k^{(n)}} 1 \\
&=& (\#A_k^{(n)}-\#B_k^{(n)})- \#C_k^{(n)}\\
&=&  {n+1 \choose k+1}  - \beta_k^{(n)} - \#C_k^{(n)}.
\end{eqnarray*}
So we have 

\begin{eqnarray}
\sum^{n-1}_{k=0}(-1)^k \left( \beta_k^{(n)}+\gamma_k^{(n)}\right) & = & \sum^{n-1}_{k=0}(-1)^k \left( {n+1 \choose k+1}-\#C_k^{(n)} \right) \label{eqn:Ck}.
\end{eqnarray}
Since $A_n^{(n)}=\{(0,1, \dots , n)\}$ and any two numbers which are next to each other are not equivalent, $B_n^{(n)}=\emptyset$. This implies $C_n^{(n)}$ is a one-point set. Thus, the right hand side of  $(\ref{eqn:Ck})$ is 
\begin{eqnarray*}
\sum^{n}_{k=-1}(-1)^k \left( {n+1 \choose k+1}-\#C_k^{(n)} \right)  &=& \sum^{n}_{k=-1}(-1)^k{n+1 \choose k+1} - \sum^{n}_{k=-1}(-1)^k\#C_k^{(n)}\\
&=& - \sum^{n}_{k=-1}(-1)^k\#C_k^{(n)}.
\end{eqnarray*}
To prove this Lemma it suffices to show that 
\begin{eqnarray}\sum^{n}_{k=-1}(-1)^k\#C_k^{(n)}=0.\label{eqn:Ck2}
\end{eqnarray}
We prove this by induction on $n$.

At $n=0$, the left hand side of $(\ref{eqn:Ck2})$ is 
\begin{eqnarray*}
\sum^{0}_{k=-1}(-1)^k \#C_k^{(0)} & = &(-1)\#C_{-1}^{(0)} + \#C_{0}^{(0)}\\
&=& -1+1\\
&=& 0.
\end{eqnarray*}

 We need some preparations for the next step. Let $n_0$ be the maximum number of all numbers which are equivalent to $n$ and less than $n$.
$$n_0 = \max \{ m \in [n] \mid m\sim n , m<n \}$$
If it does not exist, formally we let $n_0=-1$. We denote $ \bigcup^\infty_{k=-1} A_k^{(n)}$ by simply $A^{(n)}$. $B^{(n)}$ and $C^{(n)}$ are also defined in the same way. Define a map $\map{\psi_n}{A^{(n-1)}}{A^{(n)}}$ by $\psi_n(i_0,i_1,\dots, i_k)= (i_0,i_1,\dots, i_k,n)$. Then, it is clear $\psi$ is an injection. We give a lexicographic order to $A_k^{(n)}$. Then, $A_k^{(n)}$ is an well-ordered set. Define the map $\map{M_k^{(n)}}{C_k^{(n)}}{A_k^{(n)}}$ by taking the mininum element of each equivalence classes, that is, $M_k^{(n)}[(i_0,i_1,\dots, i_k)]=\min [(i_0,i_1,\dots, i_k)]$. Since $[(i_0,i_1,\dots, i_k)]$ is an non-empty subset of $A_k^{(n)}$, there certainly exists the minimum element. We denote $\map{\bigcup^\infty_{k=-1} M_k^{(n)}}{C^{(n)}}{A^{(n)}}$ by $M^{(n)}$. It is clear $M^{(n)}$ is an injection. So we count the number of elements of $\im M_k^{(n)}$ instead of the number of elements of $C_k^{(n)}$. 

Divide $\im M^{(n)}$ into disjoint sets $\im M^{(n-1)}$ and $(\psi_n(\im M^{(n-1)})\cap \im M^{(n)})$
\begin{eqnarray}
\im M^{(n)}& = & \im M^{(n-1)} \cup (\psi_n(\im M^{(n-1)})\cap \im M^{(n)})\label{n}
\end{eqnarray}
And divide $\im M^{(n-1)}$ into three disjoint sets $\im M^{(n_0-1)}$, $\im M^{(n_0)}-\im M^{(n_0-1)}$ and $\im M^{(n-1)}-\im M^{(n_0)}$. 
\begin{multline}
\im M^{(n-1)}=\im M^{(n_0-1)} \cup (\im M^{(n_0)}-\im M^{(n_0-1)} )\cup \\
(\im M^{(n-1)}-\im M^{(n_0)}) \label{udon}
\end{multline}
Then, we have
\begin{multline}
(\psi_n(\im M^{(n-1)})\cap \im M^{(n)})= \\
(\psi_n (\im M^{(n_0-1)})\cap \im M^{(n)}) \\ 
\cup (\psi_n (\im M^{(n_0)}-\im M^{(n_0-1)}) \cap \im M^{(n)})\\
 \cup (\psi_n (\im M^{(n-1)}-\im M^{(n_0)}) \cap \im M^{(n)}).\label{n-1}
\end{multline}
Here, note that $\psi_n (\im M^{(n_0-1)})\cap \im M^{(n)}$ is an empty-set. Indeed, for any $(i_0,i_1\dots, i_{k-1}, n)$ of $\psi_n (\im M^{(n_0-1)})$, we have $$(i_0,i_1\dots, i_{k-1}, n) >(i_0,i_1\dots, i_{k-1}, n_0)$$ and $$(i_0,i_1\dots, i_{k-1}, n)\approx (i_0,i_1\dots, i_{k-1}, n_0).$$ So $(i_0,i_1\dots, i_{k-1}, n)$ is not minimum in its equivalence class. Moreover, 
$$\psi_n (\im M^{(n_0)}-\im M^{(n_0-1)})\cap \im M^{(n)}$$
 is also an empty-set. Since all the elements of 
$$\psi_n (\im M^{(n_0)}-\im M^{(n_0-1)})$$ 
have the form of $(i_0,i_1,\dots, i_{k-1},n_0,n)$, they belong to $B^{(n)}$.
 Therefore, we obtain $(\ref{n-1})$ is
\begin{eqnarray*}
(\psi_n(\im M^{(n-1)})\cap \im M^{(n)}) &=&(\psi_n (\im M^{(n-1)}-\im M^{(n_0)}) \cap \im M^{(n)}) \\
&=& \psi_n (\im M^{(n-1)}-\im M^{(n_0)}).
\end{eqnarray*}
This implies $(\ref{n})$ is 
\begin{eqnarray}
\im M^{(n)}& = & \im M^{(n-1)} \cup (\psi_n(\im M^{(n-1)})- \im M^{(n_0)}).\label{finally}
\end{eqnarray}

Let $l$ be the length function defined by 
$$l(i_0,i_1,\dots, i_k)=k$$
for $(i_0,i_1,\dots, i_k) \in A^{(n)}$.

We finally start to calculate the left hand side of $(\ref{eqn:Ck2})$. We have 
\begin{eqnarray*}
\sum^{n}_{k=-1}(-1)^k\#C_k^{(n)}&=&\sum^{n}_{k=-1}(-1)^k\#\im M_k^{(n)} \\
&=&\sum^{n}_{k=-1} \sum_{x\in \im M_k^{(n)}}(-1)^{l(x)}\\
&=&\sum_{x\in \im M^{(n)}}(-1)^{l(x)}.
\end{eqnarray*}

\begin{eqnarray*}
\text{Here, $(\ref{finally})$ implies }&&\\
\sum_{x\in \im M^{(n)}}(-1)^{l(x)}&=& \sum_{x\in \im M^{(n-1)}}(-1)^{l(x)} + \sum_{x\in \im M^{(n-1)}-\im M^{(n_0)}}(-1)^{l(x)}.
\end{eqnarray*}
$(\ref{udon})$ implies the right hand side of the above is equal to 
\begin{multline}
\sum_{x\in \im M^{(n_0-1)} }(-1)^{l(x)} +\sum_{x\in \im M^{(n_0)}-\im M^{(n_0-1)} }(-1)^{l(x)} + \\
\sum_{x\in \im M^{(n-1)}-\im M^{(n_0)}}(-1)^{l(x)} + \sum_{x\in \psi_n (\im M^{(n-1)}-\im M^{(n_0)})}(-1)^{l(x)} \label{kagawa}
\end{multline}
Since $\psi_n$ is an injection, we have 
\begin{eqnarray}
\sum_{x\in \psi_n (\im M^{(n-1)}-\im M^{(n_0)})}(-1)^{l(x)+1} =
\sum_{x\in \im M^{(n-1)}-\im M^{(n_0)}}(-1)^{l(x)}. \notag
\end{eqnarray}
This implies $(\ref{kagawa})$ is equal to 
\begin{eqnarray*}
&\displaystyle \sum_{x\in \im M^{(n_0-1)} }(-1)^{l(x)} +\sum_{x\in \im M^{(n_0)}-\im M^{(n_0-1)} }(-1)^{l(x)} \\
=& \displaystyle \sum_{x\in \im M^{(n_0)} }(-1)^{l(x)} \\
=&\displaystyle  \sum_{k=-1}^{n_0}(-1)^k \# C_k^{(n_0)}.
\end{eqnarray*}
Since $n_0 < n$, the assumption of the induction completes the proof. If $n_0=-1$, we can use the same argument above as $\im M^{(-1)}$ and $\im M^{(-2)}$ are empty-sets. 
\end{proof}

\end{document}